\documentclass{amsproc}

\usepackage{amsthm,amsmath,amssymb}
\usepackage{dsfont}
\usepackage[utf8]{inputenc}
\usepackage[english]{babel}
\usepackage{mathtools} 
\usepackage{bm}
\usepackage{stmaryrd}
\usepackage{textcomp}
\usepackage{multicol}
\usepackage{subfig}
\usepackage{leftidx}
\usepackage{bigints}
\usepackage{color}

\newtheorem{theorem}{Theorem}[section]

\newtheorem*{problem}{Problem}
\newtheorem{cor}[theorem]{Corollary}

\theoremstyle{definition}
\newtheorem{definition}[theorem]{Definition}

\theoremstyle{remark}

\newcommand{\CC}{\mathds{C}}
\newcommand{\RR}{\mathds{R}}

\newcommand{\ZZ}{\mathds{Z}}
\newcommand{\QQ}{\mathds{Q}}

\newcommand{\KK}{\mathds{K}}

\DeclareMathOperator{\Card}{Card}

\DeclareMathOperator{\Sing}{Sing}

\DeclareMathOperator{\Hom}{Hom}

\numberwithin{equation}{section}

\begin{document}
	
	\title[Archimedean Zeta Functions and Oscillatory Integrals]{Archimedean Zeta Functions and Oscillatory Integrals}
	
	
	\author[Edwin Le\'on-Cardenal]{Edwin Le\'on-Cardenal}
	\address{CONACYT--Centro de Investigaci\'on en Matem\'aticas (CIMAT). Unidad Zacatecas \\
		Parque Quantum Ciudad del Conocimiento. 
		Av.~Lasec, Andador Galileo Galilei\\
		Manzana 3 Lote 7\\ 
		C.P. 98160, Zacatecas, ZAC, Mexico}
	\email{edwin.leon@cimat.mx}
	\thanks{The author was partially supported by CONACYT, Grant \# 286445.}


	\date{}
	
	\begin{abstract}
		This note is a short survey of two topics: Archimedean zeta  functions and Archimedean oscillatory integrals. We have tried to portray some of the history of the subject and some of its connections with similar devices in mathematics. We present some of the main results of the theory and at the end we discuss some generalizations of the classical objects.
	\end{abstract}
	
	\maketitle
	
	\section{Introduction}\label{Sec:Intro}
	Probably the first Archimedean zeta function was the Gamma function, introduced by Euler in a letter to Golbach in 1729 with the aim to interpolate the factorial, see \cite{Da}. 
	If the real part of the complex number $s$ is positive, then the Gamma function is defined via the convergent improper integral
	\begin{equation}\label{Eq:Gamma}
		\Gamma(s) = \int_0^\infty x^{s-1} e^{-x}\, \mathrm{d}x.
	\end{equation}
	It is not difficult to show that  $\Gamma(s)$ converges absolutely on $\Re (s) > 0$. Moreover the function admits an analytic continuation as a meromorphic function that is holomorphic in the whole complex plane except the non-positive integers, where $\Gamma(s)$ has simple poles.
	
	The Gamma function is just a particular case of the more general definition of an Archimedean zeta function. Take $\KK=\RR$ or $\CC$, and fix an open set $U$ in $\KK^n$, by $\mathcal{S}(U)$ we will denote the Schwartz space of $U$, see Section \ref{Sec:Dist_and_Loc}. Consider a $\KK$-analytic function $f : U\rightarrow \KK$ and a function $\phi\in\mathcal{S}(U)$. The \textit{local zeta function attached to the pair} $(f,\phi)$ is the parametric integral
	\begin{equation}\label{Eq:DefIntro}
		Z_{\phi}(s,f)=\int\limits_{\KK^{n}\smallsetminus f^{-1}(0)}\phi(x)\ |f(x)|_{\KK}^s\ |\mathrm{d}x|,
	\end{equation}
	for $s\in\CC$ with $\Re(s)>0$, where $|\mathrm{d}x|$ is the Haar
	measure on $\KK^{n}$. For uniformity reasons we will use for $a\in\CC$ the convention $|a|_{\KK}=||a||_{\CC}^2$, where $||a||_{\CC}$ is the standard complex norm. 
	
	It is easily seen that $Z_{\phi}(s,f)$ converges on the half plane $\{s\in\CC\mid \Re(s)>0 \}$ and defines a holomorphic function there. Furthermore, $Z_{\phi}(s,f)$ admits a meromorphic continuation to the whole complex plane. This was proved by I. N. Bernstein and S. I. Gel'fand in \cite{BerGel}, then independently by M. Atiyah in \cite{At}, both proofs make use of Hironaka's theorem on resolution of singularities \cite{Hir}. Later, I. N. Bernstein \cite{Ber72} gave a proof by using the following functional equation:
	\[P(s,x,\partial/\partial x)\cdot f(x)^{s+1} = b(s)f(x)^s.
	\]
	Here $b(s)\in\KK[s]$ and $P$ denotes a polynomial in the `variables'  $s, x_1,\ldots,x_n,\partial/\partial x_1$, $\ldots,\partial/\partial x_n$ and with coefficients in $\KK$. 
	The monic generator of the ideal of functions satisfying the functional equation is called the \textit{Bernstein-Sato polynomial of} $f$. 
	
	In fact, the theory of the Bernstein-Sato polynomials, also known as $b$-functions, was developed by M. Sato during the 60's at Kyoto under the name of algebraic theory of linear differential systems, as it is mentioned in the translation of his original work in \cite{Sat90}. This theory was framed in the theory of prehomogeneous vector spaces, see the comments in \cite[Chapter 6]{IguBook} and \cite{Gra}. Some years later, the advanced student of I. M. Gel'fand, I.N. Bernstein was thinking about similar questions in Moscow \cite{Ber68,Ber71}. Both works initiated the huge $\mathcal{D}$-module theory which has multiple and fruitful ties with algebraic geometry, singularity theory, topology of varieties, representation theory and (of course) differential equations, among others. See the nice surveys \cite{Gra,Wal,AlJeNuB} and the books \cite{Bjo,Cou}. 
	
	The other topic that we treat in this note is \textit{oscillatory integrals}, which are integrals of the form
	\[I_\phi(\tau;f)=\int_{\RR^{n}} \exp(\mathrm{i}\tau f(x))\,\phi(x)\, |\mathrm{d}x|,\]
	for $f$ real analytic, $\tau \in \RR$ and $\phi\in\mathcal{S}(U)$. 
	They can be considered as generalizations of the Fourier transform and can be traced back to the original works of Airy, Stokes, Lipschitz, and Riemann, see the nice historic remarks in the book by E. Stein \cite[Chapter VIII]{Ste}. They have played a major role in harmonic analysis, partial differential equations and number theory among others. The function $f$ is called the \textit{phase} and the function $\phi$ is called the \textit{amplitude}. Here a central question is to know (or, better, estimate) how $I_\phi(\tau;f)$ decays when the real parameter $\tau$ tends to $\infty$. A general principle known as the stationary phase principle or the saddle point method 
	states that the main contribution in the asymptotics of  $I_\phi(\tau;f)$ is given by neighbourhoods of the critical points of the phase. Elementary methods can be used to study simple phases, cf. \cite[Chapter VIII]{Ste} and \cite[Section III.4.5]{GelShi}, but more complicated phases require more elaborated methods. 
	
	The general theorem on asymptotic expansions of $I_\phi(\tau;f)$ (c.f. Theorem \ref{Thm:AsyExp}) was proved first by Jeanquartier in \cite{Jea70}, and then by Malgrange \cite{Mal74} and Igusa \cite{IguTata}. Again the necessary tool is resolution of singularities, and the conclusion is that the asymptotic behaviour is controlled by the poles of $Z_{\phi}(s,f)$.
	
	In the proof given by Igusa in \cite{IguTata}, he initiates the uniform treatment of the subject of local zeta functions over local fields $\KK$ of characteristic zero, i.e. $\KK=\RR,\, \CC,\, \QQ_p$ or a finite extension of $\QQ_p$. Since then, many authors have been working on this theory of local zeta functions, see the book by Igusa \cite{IguBook} and the references therein. The theory of $p$-adic zeta functions (also called Igusa zeta functions) has evolved greatly in the last 30 years, boosted mainly by the definition of the motivic zeta functions of J. Denef and F. Loeser in \cite{DenLoe}. For more detailed descriptions of the work and legacy of Igusa see the classic survey by J. Denef \cite{DenRepo} and the more recent survey by D. Meuser \cite{Meu}. It is not our purpose in this note to detail this whole subject, for an introduction to the theory of $p$-adic zeta functions we recommend \cite{LeoZun}, and also the recent works, including the motivic case, to appear in this Volume \cite{PoVe,Vi}. 
	
	The proof by Malgrange of Theorem \ref{Thm:AsyExp} about the asymptotic expansions of $I_\phi(\tau;f)$ brought to the game a collection of new tools of algebraic topological nature. He considered integrals of the form 
	\[\int_{\Gamma}\exp(\tau f(x))\, \phi(x) \mathrm{d}x_1\wedge\cdots\wedge\mathrm{d}x_n,
	\]
	where $\Gamma$ is a real $n$-dimensional chain lying in $\CC^n$,  $f(x)$  is a  holomorphic function on $\CC^n$  and $\phi\in\mathcal{S}(\CC^n)$. Malgrange shows that these integrals also admit asymptotic expansions, but this time related with the monodromy of $f$ at some point of $f^{-1}(0)$. It turns out that this type of integrals is able to recover the real integrals $I_\phi(\tau;f)$, see \eqref{Eq:ORIandOCI}, providing a connection between the poles of $Z_{\phi}(s,f)$ and the eigenvalues of the (complex) monodromy of $f$ at some point of $f^{-1}(0)$. This point of view contributes to enlarge and foster the theory of Archimedean zeta functions. 
	
	Finally we want to mention that there have been several generalizations of the classical Archimedean zeta functions $Z_{\phi}(s,f)$. Most of them use a variant proposed by C. Sabbah in \cite{Sab87B}, where he considers several complex variables $s_1,\ldots,s_l$ attached to the same number of functions, see Section \ref{Sec:LZFGen}. There have been also some generalizations of the oscillatory integrals $I_\phi(\tau;f)$ in the frame of multilinear (oscillatory) operators, see the work of Phong, Stein, and Sturm in \cite{PhStSt01}, and the work of Christ, Li, Tao and Thiele in \cite{ChLiTaTh}. 
	
	We now describe the content of the article. In Section \ref{Sec:ArqLZF} we present the basic definitions of the Scwartz space, the Weyl algebra and the Archimedean local zeta functions. We also present in this section our first proof of Theorem \ref{Thm:mer_cont} about the meromorphic continuation of $Z_{\phi}(s,f)$, by using resolution of singularities. We also present as an application of Theorem \ref{Thm:mer_cont} a proof of the fact that differential operators with constant coefficients admit fundamental solutions. Section \ref{Sec:BerPol} is devoted to the presentation of the Bernstein-Sato polynomial and the second proof of the meromorphic continuation of $Z_{\phi}(s,f)$ in Theorem \ref{Thm:Ber72}. We discuss also here the difference between the sets of candidate poles for the Archimedean zeta functions. Next, in Section \ref{Sec:OscInt} we introduce the oscillatory integrals, the Gel'fand Leray form and some interesting fiber integrals. We state in Theorem \ref{Thm:AsyExp} the asymptotic expansion of $I_\phi(\tau;f)$ and relate some of its parameters with the log-canonical threshold of $f$. We briefly explain in Section \ref{Sec:RealOI} some of the problems when dealing with oscillatory integrals of real analytic phases. In Section \ref{Sec:CompxOI} we review some of the results of Malgrange about complex oscillatory integrals and their relations with  $Z_{\phi}(s,f)$. Finally in Section \ref{Sec:SomeGen} we present some generalizations of Archimedean local zeta functions and oscillatory integrals.
	
	We have tried to provide an ample list of references to the classic and current literature about the considered topics. No new proofs are given and only some proofs are presented. Our purpose is not just to expose the classic theory but also to call the attention of the community to some not so well known relations among Archimedean local zeta functions and oscillatory integrals. The most important example of this fact is the almost inexistent connections between real zeta functions and multilinear oscillatory integrals.   
	\section{Archimedean Local Zeta Functions}\label{Sec:ArqLZF}
	The first serious study of real and complex zeta functions was carried out in the first volume of the book \textsc{Generalized Functions} by I. M.~Gel'fand and G.~Shilov \cite{GelShi}. They begin their study with the following \textit{regularization} technique. Take $s\in\CC$ with $\Re(s)>-1$ and let $\phi$ be a smooth real function with compact support. Note that the integral 
	\[(x_+^s,\phi)=\int_0^\infty x^s \phi(x)\,\mathrm{d}x,
	\]
	is a holomorphic function on $\Re(s) > -1$. Moreover we can write
	\begin{equation}\label{Eq:BasicExam}
		(x_+^s,\phi)=\int_0^1 x^s\,[\phi(x)-\phi(0)]\,\mathrm{d}x+\frac{\phi(0)}{s+1}+\int_1^\infty x^s\phi(x)\,\mathrm{d}x.
	\end{equation}
	The first term of the RHS of \eqref{Eq:BasicExam} is defined for $\Re(s)>-2$, the second for $s\ne -1$ and the last one for all $s\in\CC$. Hence $(x_+^s,\phi)$ can be continued analytically to $\Re(s)>-2$ and $s\neq -1$. Analogously, $(x_+^s,\phi)$ can be continued analytically (\textit{regularized}) to the region $\Re(s)>-n-1$, $s\ne -1,-2,\ldots,-n $, by means of
	\begin{gather*}
		(x_+^s,\phi)=\int_0^1 x^s\left[\phi(x)-\phi(0)-\dots-\frac{x^{n-1}}{(n-1)!}\phi^{(n-1)}(0)\right]\mathrm{d}x
		+\sum_{k=1}^n\frac{\phi^{(k-1)}(0)}{(k-1)!(s+k)}\\
		+\int_1^\infty x^s\,\phi(x)\,\mathrm{d}x.
	\end{gather*}
	Finally, $(x_+^s,\phi)$ has an analytic continuation to the whole $\CC$ except at the points $s=-1, -2,\ldots$, where it has simple poles. 
	
	A complex analogue of $(x_+^s,\phi)$ is studied in Section 3.6 of the same book. Indeed, a considerable part of this text is devoted to the generalization of $(x_+^s,\phi)$, replacing $x_+$  by more elaborated polynomials in several variables over $\RR$ or $\CC$. During a talk at the ICM of 1954 \cite{Gel}, I. M.~Gelfand posed the following problem. 
	\begin{problem}
		Let $P(x_1,\ldots,x_n)$ be a polynomial and consider the area in which $P>0$. Let $\phi(x_1,\ldots,x_n)$ be a smooth function with compact support. Then the generalized function \[(P^s,\phi)=\int_{P>0} P(x_1,\ldots,x_n)^s \phi(x_1,\ldots,x_n)\,|\mathrm{d}x_1\ldots\mathrm{d}x_n|,\]
		is a meromorphic function of $s$.
	\end{problem}
	In \cite{GelShi} the authors had discovered that this problem requires a careful analysis of the singular set of $P$, $\mathrm{Sing}_P$. While some simple singularities could be handled, the general case appeared to require new ideas. I.M.~Gelfand was aware that the solution to this problem would imply the existence of fundamental solutions for differential operators with constant coefficients, see Section \ref{Sec:FunSol}. The new idea that settled the problem definitively was resolution of singularities. By using Hironaka's theorem \cite{Hir}, I.N.~Bernstein and S.I.~Gelfand in  \cite{BerGel}, and independently M.~Atiyah in \cite{At}, reduce the general problem to the monomial case, which can be easily obtained from the example  $(x_+^s,\phi)$. In order to present their proof we introduce some notation. 
	\subsection{Distributions and local zeta functions}\label{Sec:Dist_and_Loc}
	Following Igusa \cite{IguBook}, we denote by 
	\[D_n=\CC[x,\partial/\partial x]=\CC[x_1,\ldots,x_n,\partial/\partial x_1,\ldots,\partial/\partial x_n]\]
	the $n$-th Weyl algebra over $\CC$, i.e. the ring of differential operators with polynomial coefficients in $n$ variables over $\CC$. A smooth function $\phi\,:\RR^n \to \CC$, i.e. a function having derivatives of arbitrarily high order, is called a \textit{Schwartz function} if 
	\[||P\phi||_\infty\quad \text{is finite for every}\quad P\in D_n=\CC[x,\partial/\partial x],
	\]
	where $||\phi||_\infty=\sup_{x\in\RR^n}|\phi(x)|$. The set of Schwartz functions is called the \textit{Schwartz space} $\mathcal{S}(\RR^n)$. By identifying $\CC$ with $\RR^2$, we similarly define $\mathcal{S}(\CC^n)$. In any case, $\mathcal{S}(\KK^n)$ ($\KK=\RR$ or $\CC$), forms a $\CC$-vector space. Moreover,  $\mathcal{S}(\KK^n)$ is a topological vector space whose topology is induced by the family of seminorms
	\[||\phi||_{i,j} = \sup_{x\in\KK^n}|x^i(\partial/\partial x)^j\phi(x)|,\]
	for all $i, j \in\ZZ_{\geq 0}^n$. As a metric space,  $\mathcal{S}(\KK^n)$ is complete and the continuous linear functionals on it form the space of tempered distributions $\mathcal{S}^\prime(\KK^n)$. There are also local versions of these spaces for an open set $U$ of $\KK^n$, see \cite[Sect. 2.1]{Hor}. It is well known that the space of smooth functions with compact support is a dense subspace of $\mathcal{S}(\KK^n)$, see e.g. \cite[Lemma 5.2.2]{IguBook}. 
	\begin{definition}
		The local zeta function of $f(x) \in \KK[x_1,\ldots, x_n]\setminus \KK$ and $\phi\in\mathcal{S}(\KK^n)$ is the element of $\mathcal{S}^\prime(\KK^n)$ defined as
		\[
		Z_{\phi}(s,f)=(f^s,\phi)=\int\limits_{\KK^{n}\setminus f^{-1}(0)}\phi(x)\, |f(x)|_{\KK}^s\, |\mathrm{d}x|,
		\]
		for $s\in\CC$ with $\Re(s)>0$. 
	\end{definition}
	This definition is essentially the same given by Igusa in \cite{IguCom}. Since the complex parameter $s$ appears as power in the distribution induced by $|f|_\KK$, I.M.~Gelfand used in \cite{Gel}  the terminology \textit{complex power} for $Z_{\phi}(s,f)$. Note that the definition given in the introduction by \eqref{Eq:DefIntro} is slightly more general than the previous one. 
	
	We now present our first proof of the meromorphic continuation of  $Z_{\phi}(s,f)$, given originally by \cite{BerGel} and \cite{At}. We follow closely the presentation of  \cite{IguBook}. For $f(x)\in\KK[x_1,\ldots, x_n]\setminus\KK$, we denote by $\Sing_f$ the singular locus of $f$, meaning the set of 
	roots of $f$ in $\KK^n$ at which all the partial derivatives simultaneously vanish. A \textit{resolution of singularities} of $f^{-1}(0)$ is a pair $(X,\pi)$ where $X$ is a smooth manifold and $\pi: X\to\KK^n$ is a proper map. In $X$ there is a  finite set $\mathcal{E}=\{E\}$ of closed submanifolds of $X$ of codimension $1$ with a pair of positive integers $(N_{E},\nu_{E})$ assigned to each $E$, which are called the numerical data of the resolution. The map $\pi$ has the following properties: (i) $(f\circ \pi)^{-1}(0)=\cup_{E\in \mathcal{E}}E$, and $\pi$  induces an isomorphism 
	\[X\setminus \pi^{-1}\left(  \mathrm{Sing}_{f} \right) \rightarrow \KK^n\setminus \mathrm{Sing}_{f};\]
	(ii) at every point $p$ of $X$ if $E_{1},\ldots,E_{m}$ are all the $E$ in $\mathcal{E}$ containing $p$ with local equations $y_{1},\ldots,y_{m}$ around $p$,
	then there exist local coordinates of $X$ around $p$ of the form $\left(
	y_{1},\ldots,y_{m},y_{m+1},\ldots,y_{n}\right)  $ such that
	\begin{equation}\label{Eq:localeqs}
		\left(  f\circ \pi\right)  \left(  y\right)  =\varepsilon\left(  y\right) {\displaystyle\prod\limits_{i=1}^{m}}
		y_{i}^{N_{i}},\quad 
		\pi^{\ast}\Bigg(\bigwedge\limits_{1\leq i\leq n}\mathrm{d}x_i\Bigg)=\eta\left(  y\right)  \left(
		{\displaystyle\prod\limits_{i=1}^{m}}
		y_{i}^{\nu_{i}-1}\right)
		\bigwedge\limits_{1\leq i\leq n}
		\mathrm{d}y_{i} 
	\end{equation}
	on some neighborhood of $p$, in which $\varepsilon\left(  y\right)  $,
	$\eta\left(  y\right)  $ are units of the local ring $\mathcal{O}_{p}$ of $X$
	at $p$. In particular $\cup_{E\in \mathcal{E}}E$ has normal crossings. 
	
	Now consider our integral $Z_{\phi}(s,f)$ and let $\left\vert\bigwedge\nolimits_{1\leq i\leq n}\mathrm{d}x_{i}\right\vert $ denote the measure induced by the differential form 
	$\bigwedge\nolimits_{1\leq i\leq n}\mathrm{d}x_{i}$ on $\KK^{n}$, which agrees with the Haar measure of $\KK^{n}$. Then
	\[Z_{\phi}\left(  s,f\right)  =\int\limits_{\KK^{n}\setminus f^{-1}\left(0\right)}\phi\left(  x\right)\,| f(x)|_{\KK}^{s}\left\vert\bigwedge_{1\leq i\leq n}\mathrm{d}x_{i}\right\vert.
	\]
	Note that by density, it is enough to make the proof for $\phi$ a smooth function with compact support. We use $\pi$ as a change of variables in our integral to  get 
	\[
	Z_{\phi}\left(  s,f\right)=\bigintsss\limits_{X\setminus \pi^{-1}(f^{-1}(0))} \phi(\pi(y))\,|f(\pi(y))|_\KK^{s}\,\left\vert\pi^{\ast}\Bigg(\bigwedge_{1\leq i\leq n}
	\mathrm{d}x_{i}\Bigg) (y)\right\vert.
	\]
	At every point $p$ of $X\setminus \pi^{-1}(f^{-1}(0))$ we can choose a chart $U$ such that \eqref{Eq:localeqs} holds. Moreover 
	we choose a small neighborhood $U_p$ of $p$ over which the above local coordinates are valid and the functions $\varepsilon$ and $\eta$ are locally invertible. Since the map $\pi$ is proper, we see that $C = \pi^{-1}(\mathrm{Supp}(\phi))$ is also compact and we will consider a partition of the unity associated to $C$ and subordinated to the cover $\{U_p\}$. By this we mean that for every element $U_p$ of the cover, there exists a finite set $\{\rho_{p,k}(y)\}_{k\in K}$ of smooth functions such that for every $y$ in $C$, only a finite number of $\rho_{p,k}$ are different from zero, and $\sum_{k\in K} \rho_{p,k}(y)=1$. This implies that 
	\[
	Z_{\phi}(s,f)=\sum\limits_{k\in K} \bigintssss_{U_p} \big(\phi(\pi(y))\cdot \rho_{p,k}(y)\,|\varepsilon(y)|_\KK^s\, \eta(y)\big)
	\prod\limits_{i=1}^{m}|y_{i}|_{\KK}^{N_{i}s+\nu_{i}-1}\left\vert\bigwedge\limits_{1\leq i\leq n}\mathrm{d}y_{i}\right\vert.
	\]
	In every neighborhood $U_p$ the function $\phi(\pi(y))\cdot \rho_{p,k}(y)\,|\varepsilon(y)|_\KK^s\, \eta(y)$ can be considered as a smooth function with compact support. Therefore, by using the $n$-dimensional version of the example $(x_+^s,\phi)$, see e.g. \cite[Lemma 7.3]{AG-ZV}, one gets the following theorem, stated in Igusa's version \cite[Theorem~5.4.1]{IguBook}.
	\begin{theorem}[{\cite[Introduction]{At},\cite[Theorem 1]{BerGel}}]\label{Thm:mer_cont}
		Let $f(x) \in \KK[x_1,\ldots, x_n]\setminus \KK$ and $\phi\in\mathcal{S}(\KK^n)$. Assume that  $(X,\pi)$ a resolution of singularities of $f^{-1}(0)$ with total transform $\mathcal{E}=\{E\}$. 	Let $(N_{E},\nu_{E})$ be the numerical data of  $(X,\pi)$. Then $Z_{\phi}(s,f)$ admits a meromorphic continuation to the whole complex plane with poles contained in the set 
		\begin{equation}\label{Eq:Poles}
			\bigcup_{E\in\mathcal{E}} -\frac{\nu_{E}+\ZZ_{\geq 0}}{qN_E},
		\end{equation}
		where $q=1$ if $\KK=\RR$ and $q=2$  if $\KK=\CC$. Moreover, the orders of the candidate poles are at most equal to $\min \{n, \Card(\mathcal{E})+1\}$. 
	\end{theorem}
	\subsection{Fundamental Solutions of Differential Operators}\label{Sec:FunSol}
	The proof by M. Atiyah in \cite{At} of Theorem \ref{Thm:mer_cont} contains other interesting consequences of the meromorphic continuation of $Z_{\phi}(s,f)$. For instance, one may give another proof of the following result of Hörmander and Lojasiewicz, known as the division of distributions problem. Again we follow  \cite[Sect. 5.5]{IguBook}.
	\begin{theorem}[{\cite[Corollary~1]{At}}]\label{Thm:At}
		If $f(x) \in \CC[x_1,\ldots, x_n]\setminus \CC$, there exists an element $T$ of $\mathcal{S}^\prime(\RR^n)$ satisfying $fT = 1$.
	\end{theorem}
	The proof is not difficult, and according to Igusa is an elegant one. A nice corollary of Theorem \ref{Thm:At} is the remarkable fact that there exists fundamental (or elementary) solutions for differential operators with constant coefficients as we shall see below. First we recall that for $P\in D_n=\KK[x,\partial/\partial x]$, its \textit{adjoint operator} $P^\ast$ is the image of the $\KK$-involution of $D_n$ given by $x_i^\ast=x_i$ and $(\partial/\partial x_i)^\ast=-\partial/\partial x_i$. In particular this implies that $(P_1P_2)^\ast=P_2^\ast P_1^\ast$ and $(P_1^\ast)^\ast=P_1$ for every 
	$P_1, P_2\in D_n$. 
	
	This notion of adjoint can be extended up to $\mathcal{S}^\prime(\KK^n)$. In particular one has for every $\phi\in\mathcal{S}(\RR^n)$ that 
	\[1^\ast(\phi)=\phi(0)=:\delta_0(\phi),\]
	which means that $1^\ast=\delta_0$ in $\mathcal{S}^\prime(\RR^n)$. The element $\delta_0$ of $\mathcal{S}^\prime(\RR^n)$ is known as the \textit{Dirac measure of} $\RR^n$ \textit{supported by} $0$. Now consider $f(x) \in \CC[x_1,\ldots, x_n]\setminus \CC$, say $f(x)=\sum_{k\in(\ZZ_{\geq 0})^n}c_k x^k$ for $c_k\in\CC$, and define an element $P$ of  $\CC[\partial/\partial x_1,\ldots,\partial/\partial x_n]$ (i.e.  a \textit{differential operator with constant coefficients}) by 
	\[P=\sum_k \frac{1}{(2\pi\mathrm{i})^{|k|}}c_k(\partial/\partial x)^k.\]
	It is a simple matter to show that for any $T$ in $\mathcal{S}^\prime(\KK^n)$ one has $(fT)^\ast = PT^\ast$.
	\begin{cor}[{\cite[Corollary~3]{At}}]\label{Cor:At}
		If $P$ is a differential operator with constant coefficients, then there exists an elementary solution for $P$, i.e., an element $S$ of 
		$\mathcal{S}^\prime(\RR^n)$ satisfying $PS = \delta_0$.
	\end{cor}
	\begin{proof}
		Assume that \[P=\sum_k \frac{1}{(2\pi\mathrm{i})^{|k|}}c_k(\partial/\partial x)^k,\] for some $c_k\in\CC$. Put $f(x)=\sum_{k\in\ZZ_{\geq 0}}c_k x^k$ in $\CC[x_1,\ldots, x_n]\setminus \CC$, then by Theorem \ref{Thm:At}, there is an element $T$ of $\mathcal{S}^\prime(\RR^n)$ satisfying $fT = 1$. By taking $S=T^\ast$, we have that $S$ belongs to $\mathcal{S}^\prime(\RR^n)$ and moreover $PS =(fT)^\ast = 1^\ast=\delta_0$.
	\end{proof}
	\section{The Bernstein-Sato Polynomial}\label{Sec:BerPol}
	So far our approach to Archimedean local zeta functions appeared only related with objects from analysis, we will see in this section how $Z_{\phi}(s,f)$ is connected with objects of more algebraic nature. The algebraic analysis theory or algebraic theory of linear differential systems began in the 60's with the school of M. Sato at Kyoto, see \cite{Sat90} for the translation of his work in the 60's. A similar development was taking place in Moscow, where I.N. Bernstein  was working on similar matters, see \cite{Ber68,Ber71}. The whole mathematical construction is known today as $\mathcal{D}$-module theory and it has connection in many areas of mathematics such as algebraic geometry, singularity theory, topology of varieties, representation theory and of course differential equations, among others. For more on the history of the subject and its connections we recommend the surveys \cite{Gra,Wal, AlJeNuB} or the books \cite{Bjo,Cou}.
	
	In this section we will only present the definition of the Bernstein-Sato polynomial and discuss some of the connections with $Z_{\phi}(s,f)$. Let us start by considering the Weyl algebra $D_n=\KK[x,\partial/\partial x]$. Note that $\KK(x)$ is a $D_n$-module and so is $\KK[x]$ because it is mapped to
	itself under operations of $D_n$. Moreover, we pick another variable $s$ (in addition to $x_1,\ldots,x_n$) and consider the Weyl algebra $D=\KK(s)[x,\partial/\partial x]$. Now we take an element $f(x)\in \KK[x]\setminus\{0\}$ and denote by $S$ the multiplicative subset of $\KK[x]\setminus\{0\}$ generated by $f(x)$, then it can be shown that 
	\[
	S^{-1}(\KK(s)[x])=\KK(s)[x_1,\ldots,x_n,1/f(x)],
	\]
	is a $D$-module. Igusa presents Bernstein's main Theorem in the following form
	\begin{theorem}[{\cite[Lemma~1]{Ber68}}]\label{Thm:Ber68}
		There exists an element $P_0$ of $D_n$ satisfying 
		\begin{equation}\label{Eq:Berns}
			P_0\cdot f(x) = 1.
		\end{equation}
	\end{theorem}
	Multiplying both sides of \eqref{Eq:Berns} by an arbitrary polynomial of $\KK[s]\setminus\{0\}$ we get
	\[P\cdot f(x) = b(s),\]
	for some polynomial $P$ in the `variables'  $s, x_1,\ldots,x_n,\partial/\partial x_1,\ldots,\partial/\partial x_n$ and with coefficients in $\KK$. If now we assume that $s\in\ZZ$, we obtain the following functional equation
	\begin{equation}\label{Eq:Func_Eq}
		P(s,x,\partial/\partial x)\cdot f(x)^{s+1} = b(s)f(x)^s,
	\end{equation}
	which is proved in detail in \cite{Ber72}. The set of all the polynomials $b(s)$ of $\KK[s]$ satisfying \eqref{Eq:Func_Eq} forms an ideal, and the unique monic generator of this ideal is called the \textit{Bernstein-Sato polynomial} $b_f(s)$ \textit{of} $f$. The meromorphic continuation of the real $Z_{\phi}(s,f)$ is stated by Bernstein in the following form.
	\begin{theorem}[{\cite[Theorem~1\textquotesingle]{Ber72}}]\label{Thm:Ber72}
		For a complex number $s$ with $\Re(s)>0$, consider $Z_{\phi}(s,f)$, the local zeta function of $f(x) \in \RR[x_1,\ldots, x_n]\setminus \RR$ and $\phi\in\mathcal{S}(\RR^n)$. Suppose that the Bernstein-Sato polynomial $b_f(s)$ of $f$ factors as 
		\[b_f(s)=\prod_{\lambda}(s+\lambda),\]
		and denote by $\Gamma(s)$ the Gamma function, see \eqref{Eq:Gamma}. Then $\frac{Z_{\phi}(s,f)}{\prod_{\lambda}\Gamma(s+\lambda)}$ has a holomorphic
		continuation to the whole $\CC$.
	\end{theorem}
	\begin{proof}
		In order to avoid complications with the absolute value we first consider the following simplification. Let $f_+^s$ be the function defined by 
		\[f_+^s(x)=\begin{cases}
			f(x)^s & \text{ for } f(x)\geq 0\\
			0 & \text{ for } f(x)<0.\\
		\end{cases}\]
		Then by the functional equation \eqref{Eq:Func_Eq} one has
		\begin{align*}
			b_f(s)\,Z_{\phi}(s,f_+)&=\int\limits_{\RR^{n}\setminus f^{-1}(0)}\phi(x)\,  b_f(s)f_{+}^s(x)\, |\mathrm{d}x|=\int\limits_{\RR^{n}\setminus f^{-1}(0)}\phi(x)\,  P(s)\,f_{+}^{s+1}(x)\, |\mathrm{d}x|\\
			&=\int\limits_{\RR^{n}\setminus f^{-1}(0)}P^\ast(s)\cdot\phi(x)\,  f_{+}^{s+1}(x)\, |\mathrm{d}x|.
		\end{align*}
		Here $P^\ast$ denotes the adjoint operator of $P$, and one has that $P^\ast(s)\cdot\phi(x)\in\mathcal{S}(\RR^n)$. We can repeat the
		above process and, after $r$-times, we get
		\[b_f(s+r-1)\cdots b_f(s)\,Z_{\phi}(s,f_+)=\int\limits_{\RR^{n}\setminus f^{-1}(0)}\phi_r(x)\,  f_{+}^{s+r}(x)\, |\mathrm{d}x|,
		\]
		where $\phi_r(x)=P^\ast(s + r - 1)\cdots P^\ast(s)\cdot\phi\in\mathcal{S}(\RR^n)$. Now, by a simple calculation one has that
		\[\prod_{\lambda}\Gamma(s+\lambda)\,b_f(s) \cdots b_f(s + r - 1)=\prod_{\lambda}\Gamma(s+r+\lambda),
		\]
		which in turn implies that for $s$ with $\Re(s)$ sufficiently large
		\begin{equation}\label{Eq:ProofBern}
			\frac{Z_{\phi}(s,f_+)}{\prod_{\lambda}\Gamma(s+\lambda)}=\frac{1}{\prod_{\lambda}\Gamma(s+r+\lambda)}\,\int\limits_{\RR^{n}\setminus f^{-1}(0)}\phi_r(x)\,  f_{+}^{s+r}(x)\, |\mathrm{d}x|.
		\end{equation}
		Note that, while the LHS of \eqref{Eq:ProofBern} is holomorphic when $\Re(s)>0$, the RHS is holomorphic on $\Re(s)>-r$. Since $r\in\ZZ_{\geq 0}$ is arbitrary, one concludes that $\frac{Z_{\phi}(s,f_+)}{\prod_{\lambda}\Gamma(s+\lambda)}$ has a holomorphic continuation to $\CC$.
	\end{proof}
	The previous proof can be easily adapted for the complex $Z_{\phi}(s,f)$, see \cite[Theorem~5.3.2]{IguBook}. 
	\subsection{Poles of $Z_{\phi}(s,f)$.} By Theorem \ref{Thm:Ber72} the poles of $Z_{\phi}(s,f)$ are integer translates of the roots of $b_f(s)$. This agrees with Theorem \ref{Thm:mer_cont} since the roots of the Bernstein-Sato polynomial of $f$ are negative rational numbers \cite{Kas}. Moreover one has
	\begin{theorem}[{\cite[Theorem~1]{Lic}}]\label{Thm:Ka-Li}
		If $f(x) \in \CC[x_1,\ldots, x_n]\setminus \CC$ and $(X,\pi)$ is a resolution of singularities of $f^{-1}(0)$, with the notation of Section \ref{Sec:Dist_and_Loc}, every root of $b_f(s)$ is of the form 
		\[\lambda=-\frac{\nu_E+k}{N_E}, \] 
		for some $E\in \mathcal{E}$ and some integer $k\geq 0$.
	\end{theorem}
	
	In most cases, Theorem \ref{Thm:Ka-Li} gives only a big list of candidate roots for $b_f(s)$ and the problem of the actual determination of the true roots is largely open. Of course, the same problem translates to the candidate poles of $Z_{\phi}(s,f)$, since in general, Theorem \ref{Thm:mer_cont} gives a very big list of candidate poles and the determination of the true poles is hard. For instance it is known that, when $f$ is a reduced plane curve singularity or an isolated quasi homogeneous singularity, the set of poles of the complex $Z_{\phi}(s,f)$ coincides exactly with the set $\{\lambda-r\mid \lambda 
	\text{ is a root of } b_f(s), r\in\ZZ_{\geq 0}\}$, see \cite{Loe85}. See also \cite{Bla} for some recent results about the poles of the complex $Z_{\phi}(s,f)$ in the case of curves.
	
	It is also known that the roots of the Bernstein-Sato polynomial are deeply connected with the geometry of $f$. Malgrange \cite{Mal} showed that, if $\lambda$ is a root of $b_f(s)$ then $\exp(2\pi\mathrm{i}\lambda)$ is an eigenvalue of the local monodromy of $f$ at some point of $f^{-1}(0)$, and all eigenvalues are obtained in this way, see \cite{Bar84}. We discuss some of these results in Section \ref{Sec:CompxOI}.
	
	We have started this section mentioning the work by M. Sato in the 60's. Part of his interest in the theory of $\mathcal{D}$-modules was motivated by another mathematical theory, namely prehomogeneous vector spaces. If $G$ is a connected reductive algebraic subgroup of $GL_n(\CC)$ acting transitively on the complement of
	an absolutely irreducible hypersurface $f^{-1}(0)$ in $\CC^n$, then $(G,\CC^n)$ is called a \textit{prehomogeneous vector space}. In this case one may show that $f(x)$ is a homogeneous polynomial and also a \textit{relative invariant} of $G$, that is, 
	\[ f(g\cdot x) = \nu(g)f(x),\quad\text{for all }\, g \in G,\]
	where $\nu$ belongs to $\Hom(G,\CC^\times)$. See \cite{Sat90} for the translation of the original work of M. Sato and \cite[Chap. 6]{IguBook} for a detailed exposition. This theory of prehomogeneous vector spaces contains Archimedean zeta functions of some group invariants, see e.g. \cite{Sat74,Kim82,Gra}. It may be useful to compute explicitly some zeta functions, as in the following example, given at the end of \cite[Chap. 6]{IguBook}:
	\[\int\limits_{\CC^{n^2}} \exp(-2\pi \ltrans{x} \bar{x})\, |\det(x)|_\CC^{s}\,|\mathrm{d}x|=\frac{1}{(2\pi)^{ns}}\cdot\prod_{r=1}^n\frac{\Gamma(s+r)}{\Gamma(r)}.
	\]
	\section{Oscillatory Integrals}\label{Sec:OscInt}
	Denote by $\varepsilon(x)$ the standard additive character of $\KK=\RR$ or $\CC$, which is defined as 
	\[\varepsilon(x)=\begin{cases}
		\exp(2\pi\mathrm{i}x), \text{ for } x\in\RR\\
		\exp(4\pi\mathrm{i}\Re(x)), \text{ for } x\in\CC.
	\end{cases}\]
	\begin{definition}\label{Def:OscInt}
		If $f(x) \in \KK[x_1,\ldots, x_n]\setminus \KK$ and $\phi\in\mathcal{S}(\KK^n)$, then the \textit{oscillatory integral} attached to $(f,\phi)$ is defined as 
		\[I_\phi(\tau;f)=\int\limits_{\KK^{n}\setminus f^{-1}(0)}\phi(x)\ \varepsilon(\tau\cdot f(x))\ |\mathrm{d}x|,
		\] 
		for $\tau\in\KK$.
	\end{definition}
	This type of integrals appears frequently in a great number of situations of physical and mathematical interest, see e.g. \cite[Chapter 6]{AG-ZV} and  \cite[Chapter VIII]{Ste}. More generally $f$ is considered to be a $\KK$-analytic function, which is called the \textit{phase} and the function $\phi$ is called the \textit{amplitude}. One of the main problems in the theory of oscillatory integrals is to study the asymptotic behaviour of $I_\phi(\tau;f)$ when the norm of the parameter $\tau$ tends to $\infty$. The stationary phase principle states that the main contribution in the asymptotics is given by neighbourhoods of the critical points of the phase. Moreover, this asymptotic behaviour is controlled by the poles of 
	$Z_{\phi}(s,f)$.
	
	It is also known that the integral $I_\phi(\tau;f)$ can be rewritten in terms of one dimensional integrals using the Gel'fand-Leray forms, defined
	as follows. Assume that $f$ is a $\KK$-analytic function, then the Gel'fand-Leray form $\omega_f(x,t)$ is the unique $(n-1)$-form on 
	the level hypersurface $X_t = \{x\mid f(x)=t\}$ $(t\neq0)$, which satisfies
	\[\mathrm{d}f\wedge \omega_f(x,t)=\mathrm{d}x_1\wedge\cdots\wedge\mathrm{d}x_n.\]
	If $C_f:=\{x\in\KK^n\mid \nabla f(x) =0\}$ denotes \textit{the critical set of} $f$, and one assumes that $C_f\cap\mathrm{Supp}(\phi)\subset f^{-1}(0)$, then 
	\begin{equation}\label{Eq:FibInt}
		I_\phi(\tau;f)=\int_{\KK} \varepsilon(\tau\cdot t)\Bigg(\int_{X_t}\phi(x)\cdot \omega_f(x,t) \Bigg) |\mathrm{d}t|.
	\end{equation}
	\begin{theorem}\label{Thm:AsyExp}
		Consider $f(x) \in \KK[x_1,\ldots, x_n]\setminus \KK$ and $\phi\in\mathcal{S}(\KK^n)$. If we assume that  
		\[C_f\cap\mathrm{Supp}(\phi)\subset f^{-1}(0),\]
		then the oscillatory integral $I_\phi(\tau;f)$ can be expanded in an asymptotic series of the form
		\begin{equation}\label{Eq:AsyExp}
			\sum_{\alpha}\sum_{k=0}^{n-1}S_{k,\alpha}(\phi)\tau^\alpha\,(\ln \tau)^k,\quad\text{as } |\tau|_\KK\to\infty. 
		\end{equation}
		Here the numerical coefficients $S_{k,\alpha}(\phi)$ belong to $\mathcal{S}^\prime(\KK^n)$, and the	parameter $\alpha$ runs through the arithmetic progressions of the form \eqref{Eq:Poles}, given by the poles of $Z_{\phi}(s,f)$.
	\end{theorem}
	The previous statement is a reformulation of \cite[Theorem~1.6,~Ch.~III]{IguTata}, where three objects are related: the asymptotic expansions of $I_\phi(\tau;f)$ when $|\tau|_\KK\to\infty$, the asymptotic expansions of the fiber integral 
	\[\int_{X_t}\phi(x)\cdot \omega_f(x,t),  \text{ when } |t|_\KK\to 0,\] 
	and the poles of $Z_{\phi}(s,f)$. For the proof, some functional spaces are defined for each one of the aforementioned objects and then it is shown that these spaces are homeomorphic by using the Fourier and Mellin tranforms. The proof also includes the case of $p$-adic fields. 
	
	Another relevant characteristic of the asymptotic series of $I_\phi(\tau;f)$ is the \textit{oscillation index} $\beta_f$ of $f$, which is the leading power of the parameter $\tau$ in the expansion \eqref{Eq:AsyExp}. It turns out that the negative of $\beta_f$ coincides in some cases with another relevant invariant of the singularity defined by $f$: the \textit{log canonical threshold} of $f$. It may be defined in terms of the resolution $(X,\pi)$ of $f^{-1}(0)$ as
	\[\mathrm{lct}(f):=\min_{E\in\mathcal{E}}\frac{\nu_{E}+1}{N_E},\] 
	where we have used the notation of \eqref{Eq:Poles} in Section \ref{Sec:Dist_and_Loc}. It is well known that $-\mathrm{lct}(f)$ is exactly the largest pole of $Z_{\phi}(s,f)$. 
	For properties and more information about $\mathrm{lct}(f)$ and its role in birational geometry see e.g. \cite{Mus06,Mus12}. 
	\subsection{Real Oscillatory Integrals}\label{Sec:RealOI}
	In the realm of harmonic analysis, the real oscillatory integrals of type $I_\phi(\tau;f)$ have been an essential player for a long time. They have also been the basis for the definition of differential operators and have become a classical subject in harmonic
	analysis, partial diferential equations and geometry (geometric analysis) of manifolds. For instance, in \cite{Ste} there is the following class of oscillatory integrals of the second kind (being the real $I_\phi(\tau;f)$ the ones of the first kind):
	\[T_\tau(f)(x) = \int\limits_{\RR^{n}}\exp(\mathrm{i}\tau f(x,y))\, K(x,y) g(y)\ |\mathrm{d}y|.
	\]
	Here the kernel $K(x,y)$ carries an oscillatory factor and the function $g$ satisfies some smoothness condition. In this case, the important matter are the boundedness properties of $T_\tau(f)(x)$. A key hypothesis in many cases is the analyticity of the phase $f$, in this case Theorem \ref{Thm:AsyExp} was proved first by Jeanquartier in \cite{Jea70}, and then by Malgrange \cite{Mal74}, Igusa \cite{IguTata} and Jeanquartier \cite{Jea79}, among others. See also \cite[Section 7.3.2]{AG-ZV}.
	
	An important class of real analytic phases for which the asymptotics of $I_\phi(\tau;f)$ and the boundedness of $T_\tau(f)(x)$ (for some $T_\tau$) is better understood is the class of  non-degenerate  functions with respect to their Newton polyhedron. In his seminal paper \cite{Var}, A. N. Varchenko investigates the oscillatory index of $I_\phi(\tau;f)$ as well as the asymptotic series \eqref{Eq:AsyExp} in terms of the geometry of the Newton polyhedron of the phase. Loosely speaking, Varchenko's idea is to attach a Newton polyhedron $\mathcal{NP}(f)$ to the phase function $f$ and then define a non-degeneracy condition with respect to $\mathcal{NP}(f)$. Then one may construct a toric variety associated to the Newton polyhedron, and use the well known toric resolution of singularities to give a list of parameters $\alpha$ in \eqref{Eq:AsyExp} (equivalently a list of candidate poles for $Z_{\phi}(s,f)$). Moreover, this list can be read off from the geometry of $\mathcal{NP}(f)$ and a refinement of its normal fan. For example, Varchenko shows that under some mild conditions
	\[\mathrm{lct}(f)=-\beta_f=\frac{1}{t_0},\]
	where $t_0$ is the entry of the point in $\RR^n$ given by the intersection of the diagonal of $\RR^n$ with the boundary of $\mathcal{NP}(f)$. A more detailed exposition of the original article of Varchenko is presented in \cite[Chapters 6-8]{AG-ZV}. 
	
	It is well known to the specialist that the only data of $\mathcal{NP}(f)$ that give poles of $Z_{\phi}(s,f)$ are contained in the normal fan of the polyhedron, i.e. the extra rays required for the toric resolution give rise to fake candidate poles. This list of fake poles was removed by Denef and Sargos in \cite{DenSar} and by a different method by Aroca, G\'omez-Morales and the author in \cite{ArGoLe}.
	
	The original ideas of Varchenko have been extensively used and generalized in analysis. In particular they have been used in the study of: more general oscillatory integrals, estimation of the size of sublevel sets, boundedness and decay estimates of operators, among other matters. See for instance the works
	\cite{PhStSt99,PhSt00,PhStSt01,Kar,DeNiSa,CaWaWr,GrPrTa,Ikr,Gre10A,Gre10B,CoGrPr,OkaTak,KamNos16,Xia16,Gil,GiGrXi,Gre19}, 
	and the references therein. In particular we should note that Kamimoto and Nose have extended  in \cite{KamNos16} the work of Varchenko for the case when the phase $f$ is smooth and the amplitude has a zero at a critical point of $f$. They show that the optimal rates of decay, and other related parameters, for some weighted
	oscillatory integrals can be expressed in terms of the geometry of $\mathcal{NP}(f)$ and a refinement of its normal fan. Similar results are given for the poles of the associated local zeta functions. 
	\subsection{Complex Oscillatory Integrals}\label{Sec:CompxOI}
	We now consider the case of holomorphic functions $f$. In some sense, the corresponding counterpart of the oscillatory integrals $I_\phi(\tau;f)$  of real analytic functions are integrals over vanishing cycles, see  e.g. \cite{Mal74,Bar84,AG-ZV,Pem10A,Pem10B}. We consider two holomorphic functions on $\CC^n$, 
	$f(x)$ and $\phi(x)$. The integrals of interest have the form
	\begin{equation}\label{Eq:CompxOI}
		\int_{\Gamma}\exp(\tau f(x))\, \phi(x) \mathrm{d}x_1\wedge\cdots\wedge\mathrm{d}x_n,
	\end{equation}
	where $\Gamma$ is a real $n$-dimensional chain lying in $\CC^n$, and $\tau$ is a large real parameter. These integrals also admit asymptotic expansions, but this time related with the monodromy of $f$ at some point of $f^{-1}(0)$. We only present here the main ideas of Malgrange \cite{Mal74}, following closely the presentation of \cite[Chapter 11]{AG-ZV}. We start by assuming that 
	\[f: (\CC^n,0)\rightarrow (\CC,0)\]
	has a critical point $x_0$ of finite multiplicity, in fact we will assume that $x_0$ is the origin. We will say that the $n$-dimensional chain $\Gamma$ is \textit{admissible} if on its boundary one has $\Re(f(x))<0$. This notion admits the following interpretation in terms of the Milnor fibration.
	
	Let $B_\epsilon$ be the open ball centered at the origin with radius $\epsilon$ and let $D_\eta$ be the open punctured  disk defined by the condition $\{t\in\CC\,\textbf{;}\,|t|<\eta\}$, with $0<\eta\ll\epsilon\ll1$. We also put $X=B_\epsilon\cap f^{-1}(D_\eta)$. We recall that the \textit{Milnor fibration} of $f$ at the origin is the smooth locally trivial fibration defined by
	\begin{equation}\label{Eq:MilFib}
		f\vert_{X}:\ X\rightarrow D_\eta.
	\end{equation}
	We denote by $D_\eta^-$ the set $\{t\in D_\eta\,\textbf{;}\, \Re(t)<0\}$ and by $X^-$ the set $X\cap f^{-1}(D_\eta^-)$. With these definitions, the $n$-dimensional chain $\Gamma$ is admissible if $\Gamma\subset X$ and $\partial(\Gamma)\subset X^-$. Two admissible chains $\Gamma_1,\Gamma_2$ are said to be equivalent  if there exists an $(n + 1)$-dimensional chain $V \subset X$ such that 
	\[(\Gamma_1-\Gamma_2+\partial V)\subset X^-.\]
	The equivalence classes of admissible chains with the operations of addition and
	multiplication by scalars form a vector space, by definition coinciding with the homology group $H_n(X,X^-)$. 
	
	Since $f\vert_{X^-}:\ X^-\rightarrow D_\eta^-$ is also a trivial fibration, one may show that the homology of $X^-$ is isomorphic to the homology of  an arbitrary fiber over $D_\eta^-$. This means that for any $t\in D_\eta^-$, there is an isomorphism
	\[\partial_t\ :\ H_n(X,X^-)\rightarrow H_{n-1}(X_t),\]
	where $X_t$ is a Milnor fiber (under \eqref{Eq:MilFib}) and $H_{n-1}(X_t)$ is the reduced homology group. Now, for an admissible chain $\Gamma$, representing the class $[\Gamma]\in H_{n}(X,X^-)$, one may contract its boundary in $X^-$ to the fiber $X_t$, giving a cycle that represents the element $\partial_t[\Gamma]\in H_{n-1}(X_t)$. The family of classes obtained in this way depends continuously on $t\in D_\eta^-$. 
	\begin{theorem}[{\cite[Theorem~7.1]{Mal74}}]\label{Thm:AG-ZV1}
		Let $\omega$ be a holomorphic differential $n$-form defined on $X$, and take $[\Gamma]\in H_{n}(X,X^-)$. Then there exists a small positive number $t_0\in D_\eta$ ,such that
		\[\int_{\Gamma}\exp(\tau f)\cdot\omega \approx \int_{0}^{t_0}\exp(-\tau t)\Bigg(\int_{\partial_{-t}[\Gamma]}\, \omega_f(x,t)\Bigg)\mathrm{d}t.\]
		The sign $\approx$ means that the integrals are equal up to a negative integer power of the parameter $\tau$ when $\tau \to \infty$.
	\end{theorem}
	Note the similarity with \eqref{Eq:FibInt}. Once we have reached this point, it only rests to state the main Theorem of Malgrange in \cite{Mal74}. 
	\begin{theorem}\label{Thm:AG-ZV2}
		Let $\omega$ be a holomorphic differential $n$-form defined on $X$, and take $[\Gamma]\in H_{n}(X,X^-)$. Then the integral $\int_{\Gamma}\exp(\tau f)\cdot\omega$ can be expanded in an asymptotic series of the form
		\[\sum_{\alpha,k}S_{k,\alpha} \tau^\alpha\,(\ln \tau)^k,\quad\text{as } \tau\to\infty.\]
		Here the parameter $\alpha$ runs through a finite set of arithmetic progressions, depending only on the phase and consisting of negative rational numbers.
		Moreover, each number $\alpha$ verifies that $\exp (-2\pi\mathrm{i}\alpha)$ is an eigenvalue of the classical monodromy operator of the critical point of the phase. The coefficient $S_{k,\alpha}$ is zero when the monodromy does not have Jordan blocks of dimension greater than or equal to $k + 1$ associated with the eigenvalue $\exp (-2\pi\mathrm{i}\alpha)$. 
	\end{theorem}
	Malgrange also shows that integrals of type \eqref{Eq:CompxOI} allow one to recover some real oscillatory integrals, see e.g. \cite[Section~7]{Mal74} and \cite[Section~11.2]{AG-ZV}. If $f$ is a real analytic function and $\phi$ is supported on a small neighbourhood of the origin, then under some mild conditions on $\phi$, there exists an admissible $n$-chain $\Gamma$ for which
	\begin{equation}\label{Eq:ORIandOCI}
		\int_{\RR^{n}} \exp(\mathrm{i}\tau f)\,\phi\, \mathrm{d}x_1\wedge\cdots\wedge\mathrm{d}x_n\approx
		\int_{\Gamma}\exp(\mathrm{i}\tau f)\cdot \phi\, \mathrm{d}x_1\wedge\cdots\wedge\mathrm{d}x_n.
	\end{equation}
	From this observation follows the connection between the poles of Archimedean local zeta functions and eigenvalues of the complex monodromy of $f$ at some point of $f^{-1}(0)$.
	
	For some other works related with complex oscillatory integrals or with the ideas of Malgrange in \cite{Mal74} see for example \cite{WonMcC,Bar82,Bar84,Loe85,Pal,Bar86,Loe87,AG-ZV,BarMai89,DenRepo,BarMai93,LioRol,Jac00A,LiaPar,DelHow,Bar03,BarJed,Ang,Del,Web,Col,Pem10A,Pem10B,BraSeb,ClCoMi,Sab18,AndPet}, and the references therein.

	\section{Some generalizations}\label{Sec:SomeGen}
	In this last section we want to mention some generalizations of local zeta functions and oscillatory integrals. 
	\subsection{Local Zeta Functions}\label{Sec:LZFGen}
	In 1987 C. Sabbah introduces in \cite{Sab87B} the following generalization of $Z_{\phi}(s,f)$. Let $f_1,\ldots,f_l$ be a set of analytic functions defined over a complex (smooth) variety $X$, and let $\phi$ be a smooth function with compact support
	contained in $X$. Then the \textit{multivariate local zeta function associated to} $\phi$ and $f= (f_1,\ldots,f_l)$ is defined as
	\[Z_{\phi}(s_1,\ldots,s_l,f)=\int_{X}\phi(x)\ |f_1(x)|^{s_1}\cdots |f_l(x)|^{s_l}\ \mathrm{d}x,\]
	where $s_i\in\mathbb{C}$ with $\Re(s_i)>0$, for $i=1,\ldots,l$. 
	
	This integral converges on the subspace of $\CC^l$ defined by $\Re(s_i)>0$, for $i=1,\ldots,l$. Moreover, the author shows that there is a meromorphic continuation to the whole $\CC^l$ with poles contained in the union of some hyperplanes described in terms of some Bernstein polynomial. He shows that there is a finite set 
	$\mathcal{L}$ of linear forms with positive integer coefficients, relatively prime, such that for $i=1,\ldots,l$ the following functional equations holds:
	\[\Bigg[\prod_{L\in\mathcal{L}}b_{L,i}(L(s))\Bigg]\cdot f^s=P_k(s,x,\partial_x)f^s\cdot f_k.\]
	Here the $b_{L,i}$ denote certain Bernstein polynomials defined previously by the author in \cite{Sab87A}. The $P_k(s,x,\partial_x)$ are certain differential operators analogous to the ones defined in Section \ref{Sec:BerPol}. The set $\mathcal{L}$ is called the \textit{set of slopes} and in principle depends on the choice of some filtration of the $\mathcal{D}$-module associated with $f$.
	
	Apparently at the same time, F. Loeser defines in \cite{Loe89} similar zeta functions. He works over a local field $\KK$ of characteristic zero. In this case $f_1,\ldots,f_l$ are polynomials vanishing at the origin when $\KK$ is non-Archimedean, and analytic functions in the Archimedean case. In the $p$-adic case the set of test functions is given by locally constant functions with compact support, while in the Archimedean case it is the set of smooth functions with compact support. 
	The \textit{multivariate local zeta function} associated to a test function $\phi$ and $F= (f_1,\ldots,f_l)$ is defined as
	\begin{equation}\label{Eq:MultivarZeta}
		Z_{\phi}(s_1,\ldots,s_l,f)=\int_{\KK^n}\phi(x)\ |f_1(x)|_\KK^{s_1}\cdots |f_l(x)|_\KK^{s_l}\ |\mathrm{d}x|,
	\end{equation}
	where $s_i\in\mathbb{C}$ with $\Re(s_i)>0$, for $i=1,\ldots,l$. Actually, his definition is a little more general, since it includes the quasicharacters of $\KK$. He uses Hironaka's theorem to show that the poles of the meromorphic continuation to $\CC^l$ (as a rational function in the $p$-adic case) are contained in the union of some hyperplanes described this time in terms of the resolution of $\cup_{i=1}^lf_i^{-1}(0)$. He also shows that the set of slopes is connected with the geometry of the morphism $F$. More precisely, if one imposes some finiteness condition on $F$, the set of slopes is contained in the set of normal vectors to the faces of the Newton polyhedron of the discriminant of $F$ at the origin. 
	
	Later on, B. Lichtin begins in \cite{LichCompoI} a systematic study of (twisted) multivariate zeta functions essentially of the form \eqref{Eq:MultivarZeta}. He was motivated by Igusa's question about the behaviour of the singular series determined by polynomial mappings $P=(P_1,\ldots,P_l)\,:\KK^n \to \KK^l$, $(l\leq n)$. He found out that this behaviour depends crucially on the geometric features of the singular locus of $P$. In the $p$-adic case his results give also non trivial uniform decay estimates for generalized exponential sums attached to $P$, \cite{LichCompoI,LichBonn,LichCompoII}. In the case $\KK=\RR$ he also found in \cite{LichCrelle} non trivial uniform bounds for some oscillatory integrals of the form 
	\[
	\int\limits_{\RR^n} \exp(\mathrm{i}\{\tau_1 P_1(x)+\cdots+\tau_l P_l(x)\}) \,\phi(x) \, |\mathrm{d}x|,
	\]
	as $|(\tau_1,\ldots,\tau_l)|\to\infty$. For other applications of these techniques see \cite{LichForum,LichCompoIII}.
	
	 These multivariate zeta functions have been used recently in mathematical physics to study string amplitudes in \cite{B-GVZ-G} and Log-Coulomb Gases in \cite{Z-GZ-LL-C}. See the references given there for other uses in $p$-adic mathematical physics.
	
	Another local zeta function involving several functions was studied in a joint work with W. Veys and W. A.  Z\'{u}\~{n}iga-Galindo \cite{Le-Ve-Zu}, were we considered Archimedean zeta functions for analytic mappings. If  $\KK=\RR$ or $\CC$, let $\boldsymbol{f}=$ $\left(f_{1},\ldots,f_{l}\right)  :U\rightarrow \KK^{l}$ be a $\KK$-\textit{analytic mapping} defined on an open $U$ in $\KK^{n}$. Let $\phi:U\rightarrow\mathbb{C}$ be a
	smooth function on $U$ with compact support. Then the local zeta function
	attached to $\left(  \boldsymbol{f},\Phi\right)  $ is defined as
	\[
	Z_{\phi}(s,\boldsymbol{f})=\int\limits_{\KK^{n}\setminus\boldsymbol{f}%
		^{-1}\left(  0\right)  }\phi\left(  x\right)\,  \left\vert \boldsymbol{f}%
	(x)\right\vert _{\KK}^{s}\ |\mathrm{d}x|,\quad \text{ for } s\in\mathbb{C}\text{ with } \Re(s)>0.
	\]
	Here $|\boldsymbol{f}(x)|_{\KK}=|(f_{1}(x),\ldots,f_{l}(x)|_{\KK}$ stands for $\sqrt{\sum f_{i}^{2}(x)}$ or  $\sqrt{\sum |f_{i}(x)|_\CC^{2}}$ depending on whether $\KK$
	is $\RR$ or $\CC$,  and $|\mathrm{d}x|$ is the Haar measure on $\KK^{n}$. In this case we give a description of the possible poles of $Z_{\phi}(s,\boldsymbol{f})$ in terms of a log-principalization of the ideal $\mathcal{I}_{\boldsymbol{f}}=\langle f_{1},\ldots,f_{l}\rangle$. We also generalize some of the results of Varchenko presented in Section \ref{Sec:RealOI}, by introducing a new non-degeneracy condition associated to the Newton polyhedron of $\boldsymbol{f}$ (which in general differs from the Minkowski sum of the $\mathcal{NP}(f_i)$). 
	
	Very recently W. Veys and W. A.  Z\'{u}\~{n}iga-Galindo extend in \cite{Ve-Zu} part of the classical theory of local zeta functions and oscillatory integrals to the case of meromorphic functions defined over a local field $\KK$ of characteristic zero. Let $f, g\ :U\to\KK$ be two $\KK$-analytic functions defined on an open set  $U\subseteq \KK^n$, such that $f/g$ is not constant. The \textit{local zeta function attached to a test function} $\phi$ \textit{and} $f/g$ is defined as
	\[Z_{\phi}(s,f/g)=\int_{U\setminus D_\KK}\phi(x)\ \left\vert \frac{f(x)}{g(x)}\right\vert _{\KK}^{s}\ |\mathrm{d}x|,
	\]
	where $s$ is a complex number and $D_\KK=f^{-1}(0)\cup g^{-1}(0)$. In this case a resolution of singularities of $D_\KK$ is required even to show that $Z_{\phi}(s,f/g)$ converges in some open strip $(\alpha,\beta)$, but it turns out that $\alpha$ and $\beta$ are independent of the chosen resolution. The authors also consider oscillatory integrals, with the notation of Definition \ref{Def:OscInt}, they have the form 
	\[\int\limits_{\KK^{n}\setminus D_\KK}\phi(x)\ \varepsilon\Bigg(\tau\cdot \frac{f(x)}{g(x)}\Bigg)\ |\mathrm{d}x|.
	\]
	They show that the analogue of Theorem \ref{Thm:AsyExp} implies two different asymptotic expansions: one when the norm of the parameter tends to infinity, and another one when the norm of the parameter tends to zero. The first asymptotic expansion is controlled by some poles of the local zeta functions associated to the meromorphic functions $f/g-c$, for certain special values $c$, while the second expansion is controlled by other poles of $Z_{\phi}(s,f/g)$.
	\subsection{Oscillatory Integrals}
	Some generalizations of oscillatory integrals have been considered under the name of  multilinear oscillatory operators, see e.g. \cite{CheLu,PhStSt01,LuTao,LuWu,ChLiTaTh,Gre08,GiGrXi,NiOnZe} and the references given there. We discuss just a couple of examples to see the kind of problems of interest in the area.
	
	In \cite{PhStSt01} Phong, Stein, and Sturm investigated the multilinear oscillatory integral operator
	\[
	T_D(\tau;g_1,\ldots,g_n)=\int\limits_{D} \exp(\mathrm{i}\tau\cdot f(x_1,\ldots,x_n))\,g_1(x_1)\cdots g_n(x_n)\, |\mathrm{d}x_1\wedge\cdots\wedge\mathrm{d}x_n|,
	\]
	where $\tau$ is a real parameter, the phase function $f$ is a polynomial and $D$ is a subset of the unit ball in $\RR^n$. In addition, the functions $g_i$ belong to the 
	functional space $\mathbf{L}^{p_i}[0,1]$. One of their main results is the following \textit{decay estimate} for $n\geq 2$,
	\[|T_D(\tau;g_1,\ldots,g_n)|\leq C\, |\tau|^{-1/\alpha}\,[\ln (2+|\tau|)]^{n-1/2}\,\prod_{i=1}^{n}||g_i||_{p_i},
	\]
	where $\alpha$ is some constant depending on a certain Newton polyhedron of $f$ and $C$ is a constant that does not depend on $\tau$.
	
	A more general multilinear oscillatory integral operator was proposed by Christ, Li, Tao and Thiele in \cite{ChLiTaTh} by 
	\[
	I(\tau;g_1,\ldots,g_m)=\int\limits_{\RR^{n}} \exp(\mathrm{i}\tau\cdot f(x))\,\prod_{j=1}^{m}g_j(\pi_j(x))\,\phi(x)\,|\mathrm{d}x|.
	\]
	Here $\tau$ is a real parameter, $\phi$ is a smooth function with compact support and $f\,:\,\RR^n\rightarrow \RR$ is a measurable function. Moreover, each $\pi_j$ is a projection from $\RR^n$ to some $\RR^{k_j}$, with $1\leq k_j\leq n-1$, and the functions $g_j :\,\RR^{k_j}\rightarrow \CC$ are locally integrable with respect to the measure of $\RR^{k_j}$. 
	
	The integral $I(\tau;g_1,\ldots,g_m)$ is well defined if all the $g_j$ belong to the functional space $\mathbf{L}^\infty$, and in this case 
	\[|I(\tau;g_1,\ldots,g_m)|\leq C\,\prod_j||g_j||_{\infty}.\]
	The authors show that if $f$ is a polynomial of bounded degree and the $k_j$'s are $n-1$ or $1$ according to some conditions, then it is possible to characterize the decay rates of $I(\tau;g_1,\ldots,g_m)$. Some other cases are considered in \cite{NiOnZe} and \cite{GiGrXi}, but the general case is largely open. 
	
	Finally we want to comment that to the best of our knowledge, there are no well established connections between these multilinear oscillatory integrals and local zeta functions. In seems also that there are no `multilinear' analogues of the complex oscillatory integrals of section \ref{Sec:OscInt}. To establish such connections/analogues may be of some interest for analysts or algebraic geometers. 
	
	\section*{Acknowledgements}
	The author wishes to express his gratitude to the referee for several helpful comments and suggestions. The author also wishes to thank to Professor B. Lichtin for drawing the author’s attention to his work on multivariate local zeta functions and oscillatory integrals \cite{LichCompoI,LichBonn,LichCompoII,LichCompoIII,LichCrelle,LichForum}.
	\bibliographystyle{amsplain}
\providecommand{\bysame}{\leavevmode\hbox to3em{\hrulefill}\thinspace}
\providecommand{\MR}{\relax\ifhmode\unskip\space\fi MR }
\providecommand{\MRhref}[2]{%
	\href{http://www.ams.org/mathscinet-getitem?mr=#1}{#2}
}
\providecommand{\href}[2]{#2}

\end{document}